\documentclass[12pt]{article}
\usepackage{graphicx}
\usepackage{amssymb, amsmath, amsthm}

\newtheorem*{conjecture}{Conjecture}
\newtheorem{thm}{Theorem}
\newtheorem{lem}[thm]{Lemma}

\begin{document}

\title{The linear complexity  of new binary cyclotomic sequences of period $p^n$}
\author{Vladimir Edemskiy}

\maketitle

\begin{abstract}
In this paper, we determine the linear complexity  of a class of new
binary cyclotomic sequences of period $p^n$ constructed by  Z. Xiao
et al. (Des. Codes Cryptogr. DOI 10.1007/s10623-017-0408-7) and
prove their conjecture about high linear complexity of these
sequences.


\end{abstract}

\section{Introduction}
\label{intro}

The linear complexity is an important characteristic of
pseudo-random sequence significant for  stream ciphers. It is
defined as the length of the shortest linear feedback shift register
that can generate the sequence \cite{LN}. According to the
Berlekamp-Massey (B-M) algorithm, it is reasonable to suggest that
the linear complexity of a ``good" sequence should be at least a
half of the period \cite{LN}.

A generalized cyclotomy with respect to $N=pq$, where $p,q$ are odd
primes, was introduced by Whiteman \cite{W} and was extended  with
respect to odd integers in \cite{DH1}.  Whiteman generalized
cyclotomy is not consistent with classical cyclotomy. New
generalized cyclotomy which includes classical cyclotomy as a
special case, was introduced by Ding and Helleseth in \cite{DH2}.
The unified  approach to determine both of Whiteman and
Ding-Helleseth generalized cyclotomy  was  proposed in \cite{FG}.
Another approach was presented in \cite{ZCTY}, where the order of
the generalized cyclotomic classes depends on the choice of
parameters (see Section 2).

 Using classical
cyclotomic classes and generalized cyclotomic classes to construct
sequences is a known method for design sequences  with a high linear
complexity \cite{CDR}. There are many papers study the linear
complexity of the binary and non-binary generalized cyclotomic
sequences of period $p^n$ \cite{CM, DC, E,KS, WCD, YLX}(see also
references therein).

The linear complexity of new binary sequences of period $p^2$ for an
odd prime $p$  constructed based on the generalized cyclotomy
introduced in \cite{ZCTY} was studied in \cite{ZX}. In conclusion,
authors of this paper made a conjecture about the linear complexity
of the considered sequences with period $p^n$. Here, we will try to
prove the conjecture from \cite{ZX}.

The remainder of this paper is organized as follows. In Section 2 we
recall a definition  of new generalized cyclotomic sequences. The
linear complexity of these sequences is studied in Section 3.

\section{Preliminaries}
In this section, we first briefly recall the definition of
generalized cyclotomic classes with respect to $p^j$ for any integer
$j\geq 1$, and definition of  new generalized cyclotomic sequences
from \cite{ZX}. Throughout this paper $\mathbb{Z}_N$ will denote the
ring of integers modulo $N$, where $N$ is a positive integer. We
will also let $\mathbb{Z}_N^{*}$ denote  the multiplicative group
consisting of all invertible elements in $\mathbb{Z}_N$.

Let $p$ be an odd prime and $p=ef+1$, where $e,f$ are positive
integers. Denote $g$  a primitive root modulo $p^2$. Then $g$ is
also a primitive root modulo $p^j$ for each integer $j\geq 1$
\cite{IR}. It is well known that an order $g$ by modulo $p^j$ is
equal to $\varphi(p^j)$, where $\varphi(\cdot)$ is the Euler's
totient function.  Denote $d_j=\varphi(p^j)/e, j=1,2,\dots,n$ and
define
\begin{multline}
\label{eq1}
D^{(p^j )}_0 = \{g^{d_ j t}~(\bmod~ p^j)  : 0\leq t < e\}, \text {and} \\
D^{(p^j )}_ i = g^i D^{(p^j )}_0 = \{g^i x~(\bmod~ p^j) : x \in
D^{(p^j )}_0\}, i = 1, 2, \dots , d_j - 1.
\end{multline}
Here and hereafter,  we denote by $x \pmod N$  the least
 nonnegative integer that is congruent to $x$ modulo $N$. Cosets $D^{(p^j )}_ i$ are call generalized cyclotomic classes of
order $d_ j$ with respect to $p^j$. Since $g^{p^{l-1}(p-1)} \equiv
1~(\bmod~ p^l)$, it follows from the definition $D^{(p^j )}_ i$ that
\begin{equation}
\label{eq2} D^{(p^j )}_ i~(\bmod~p^l)=D^{(p^l )}_{i~(\bmod~d_l)}
\text{ for } l : 0<l \leq j.
\end{equation}

By \cite{ZCTY} we have that $\{D^{(p^j )}_ 0, D^{(p^j )}_
1,\dots,D^{(p^j)}_ {d_j-1}\}$ forms a partition of
$\mathbb{Z}^{*}_{p^j}$ and  for an integer $m\geq1$
$$ \mathbb{Z}_{p^m}=\bigcup_{j=1}^{m} \bigcup_{i=0}^{d_j-1}
p^{m-j}D^{(p^j )}_ i \cup \{0\}.$$

As in \cite{ZX} we take $f = 2^r$ with $r\geq 1$ and let $b$ be an
integer with $0\leq b \leq p^{n-1} f-1$. We now define two family of
sets for $m= 1,2,\dots, n$
\begin{equation}
\label{eq3} \mathcal{C}_0^{(p^m)}=\bigcup_{j=1}^{m}
\bigcup_{i=d_j/2}^{d_j-1} p^{n-j}D^{(p^j )}_ {(i+b)\pmod{d_j}}\text{
 and  } \mathcal{C}_1^{(p^m)}=\bigcup_{j=1}^{m}
\bigcup_{i=0}^{d_j/2-1} p^{m-j}D^{(p^j )}_ {(i+b)~(\bmod {d_j})}\cup
\{0\}. \end{equation}
 These are
 obvious facts that $\mathbb{Z}_{p^m}=\mathcal{C}_0^{(p^m)} \cup
\mathcal{C}_1^{(p^m)}$ and $|\mathcal{C}_1^{(p^m)}|=(p^m+1)/2$. Then
we define the  generalized cyclotomic binary sequence
$s^\infty=(s_0,s_1,s_2,\dots)$ with period $p^n$ by
\begin{equation}
\label{eq4} s_i =\begin{cases}
 0,&\text{ if  }   i~(\bmod~p^n)  \in \mathcal{C}_0^{(p^n)}; \\
  1,&\text{ if }  \hspace{4pt} i~(\bmod~p^n)  \in \mathcal{C}_1^{(p^n)}. \\
 \end{cases}
\end{equation}
Here $\mathcal{C}_1^{(p^n)}$ is the characteristic set of the
sequence $s^\infty$. The linear complexity and other characteristics
of these sequences were discussed in \cite{ZX} for $n=2$  and
authors  made a conjecture about the linear complexity of
$s^\infty=(s_0,s_1,s_2,\dots)$ for any $n$ and $p: \ 2^{(p-1)} \not
\equiv 1~(\bmod~p^2)$ in conclusion.
\begin{conjecture} [\cite{ZX}]
 Let
$s^\infty$ be a generalized cyclotomic binary sequence of period
$p^n$ defined by \eqref{eq4}. If  $2^{(p-1)} \not \equiv
1~(\bmod~p^2)$, then
$$ L =\begin{cases}
 p^n-\frac{p-1}{2}-\delta\left (\frac{p^n+1}{2}\right) ,&\text{ if  }   2 \in D^{(p)}_0, \\
p^n-\delta\left (\frac{p^n+1}{2}\right) ,&\text{ if  }   2 \not \in D^{(p)}_0, \\
 \end{cases}$$
where $\delta(t) = 1$ if $t$ is even and $\delta(t) = 0$ if $t$ is
odd. \end{conjecture}
 Our goal is a proof of this
conjecture. We need some preliminary notation and results before we
begin.

\section{Linear complexity}

It is well known that if $s^\infty=(s_0,s_1,s_2,\dots)$  is a
sequence of period $N$, then
 the linear complexity of this
sequence is defined by
 $$  L=N-\deg  \gcd\bigl (x^{N}-1,S(x)\bigr ),
$$
where  $S(x) = s_0 + s_1x + ... + s_{N-1}x^{N-1}$ (see, for example
\cite{CDR}). Sometimes, $S(x)$ is called the generating polynomial
of sequence $s^\infty$. One way to find the linear complexity of a
sequence is to examine roots of $S(x)$. Let $\beta $ be a primitive
$N$-th root of unity in the extension of $\mathbb{ F}_2$ (the finite
field of order two). Then latter formula can be rewritten as
\begin{equation} \label{eq5}
L=N-\left|\left\{i \left| S(\beta^i)=0,\ \ i=0,1,\dots,N-1
\right.\right\}\right|.
\end{equation}
In this section, we shall prove the conjecture about the linear
complexity of  new generalized cyclotomic binary sequences of period
$p^n$ defined in \eqref{eq4}. In the following subsections we divide
the proof of the conjecture into  several lemmas.

\subsection{Subsidiary lemmas}
 In the following two propositions, we present some properties related to $2~(\bmod~p^j)$,
 whose validity  are  easily verified  from the definition and basic number theory.
\begin{lem}
\label{l1} Let $ 2^{(p-1)}  \not \equiv 1~(\bmod~p^2)$ and let $c$
be  the least positive exponent such that $ 2^c  \equiv
1~(\bmod~p)$. Then $2^k \equiv 1~(\bmod~p^{j+1})$ for $j\geq 1$ if
and only if $k \equiv 0~(\bmod~cp^j)$.
\end{lem}
\begin{proof}
  It is clear that $c|k$. Let $k=cdp^l$, where $
\gcd(p,d)=1$ and $l\geq 0$. By condition of Lemma \ref{l1}, there
exist $t: 2^c=1+pt, \gcd(p,t)=1$. Using  properties of binomial
coefficients we obtain that $2^{cd}\equiv 1+dtp~(\bmod~p^{2})$ and
$2^{cdp^l}\equiv 1+dtp^{l+1}~(\bmod~p^{l+2})$. Hence, if $2^k \equiv
1~(\bmod~p^{j+1})$ then $l\geq j$. To conclude the proof, it remains
to note that $2^{cp^j}\equiv 1+cp^{j+1}~(\bmod~p^{j+2})$.
\end{proof}

\begin{lem}
\label{l2} Let $ 2^{(p-1)}  \not \equiv 1~(\bmod~p^2)$ and $ 2
\equiv g^u~(\bmod~p^2)$. Then  $u \not \equiv 0~(\bmod~p)$.
\end{lem}
\begin{proof}
 Suppose  $ u  \equiv
0~(\bmod~p)$; then  $g^{u(p-1)} \equiv 1~(\bmod~p^2)$ and $
2^{(p-1)} \equiv 1~(\bmod~p^2)$. This contradicts  the condition of
Lemma \ref{l2}.
\end{proof}

Let $\bar{\mathbb{F}}_2$ be an algebraic closure of $\mathbb{F}_2$
and let $\alpha_n \in \bar{\mathbb{F}}_q$  be a primitive $p^n$-th
root of unity. Denote $\alpha_j=\alpha_n^{p^{n-j}}, j=1,2\dots,n-1$.
Then $\alpha_j$ is a primitive $p^j$-th root of unity in an
extension field of $\mathbb{F}_2$. As usual,
$\mathbb{F}_2(\alpha_j)$ denote a simple extension of $\mathbb{F}_2$
obtained by adjoining an algebraic element $\alpha_j$ \cite{LN}. The
dimension of the vector space $\mathbb{F}_2(\alpha_j)$ over
$\mathbb{F}_2$ is then called the degree of $\mathbb{F}_2(\alpha_j)$
over $\mathbb{F}_2$, in symbols $[\mathbb{F}_2(\alpha_j):
\mathbb{F}_2]$.
\begin{lem}
\label{l3}
Let $ 2^{(p-1)}  \not \equiv 1~(\bmod~p^2)$. Then
$$[\mathbb{F}_2(\alpha_{j+1}): \mathbb{F}_2(\alpha_j)]=p$$
for $j=1,2,\dots,n-1.$
\end{lem}
\begin{proof}  It is well known that if $K$ is a finite extension of $\mathbb{F}_2$
then $|K|=2^{[K : \mathbb{F}_2]}$ and  an order of any nonzero
element $K$  divides $|K|-1$ \cite{LN}.

As in Lemma \ref{l1}, let $c$ be a least positive integer such that
$ 2^c \equiv 1~(\bmod~p)$. According to Lemma \ref{l1} we have that
$ 2^{cp^j} -1 \equiv 0~(\bmod~p^j)$, hence $\alpha_j \in
\mathbb{F}_{2^{cp^j}}$ and  $[\mathbb{F}_2(\alpha_j) :
\mathbb{F}_{2}]= cp^j$. With similar arguments we get
$[\mathbb{F}_2(\alpha_{j+1}) : \mathbb{F}_{2}]= cp^{j+1}$. The
proposition of this lemma is now established.
\end{proof}

\subsection{Properties of sequence generating polynomial}
 Let us introduce the auxiliary polynomials $E^{(p^j)}_l(x)=\sum_{i \in D^{(p^j)}_l} x^i, j=1,2,\dots,n; l=0,\dots,
d_j-1$. Denote
 $$H^{(p^j)}_{k~(\bmod~d_j)}(x)=\sum_{i=0}^{ d_j/2-1}
E^{(p^j)}_{(i+k)~(\bmod~d_j)} (x), $$
 and
$$T^{(p^j)}_{k~(\bmod~d_j)}(x)=H^{(p^j)}_{k~(\bmod~d_j)}(x)+H^{(p^{j-1})}_{k~(\bmod~d_{j-1})}(x^p)+\dots
+H^{(p)}_{k~(\bmod~d_{1})}(x^{p^{j-1}}),$$ $k=0,1,\dots, d_j-1.$

Then $H^{(p^j)}_{k~(\bmod~d_j)}(x)=\sum_{i\in \bigcup_{t=0}^{d_j-1}
D^{(p^j )}_ {{t+k}~(\bmod~d_j)}} x^i$,
$T^{(p^j)}_{k~(\bmod~d_j)}(x)=\sum_{i \in \mathcal{C}_1^{(p^j)}}
x^i$ and $S(x)=T^{(p^n)}_b(x)+1$ by \eqref{eq3} and \eqref{eq4}.

 In the rest of this paper, the subscripts $i$ in
$D^{(p^j)}_i$ and $H^{(p^j)}_i(x), T^{(p^j)}_i(x)$ will be always
assumed to be taken modulo the order $d_j$, and the modulo operation
will be omitted for simplicity.

Using \eqref{eq5} we get that
\begin{equation}
\label{eq6} L=N-\left|\left\{i \left| T^{(p^n)}_b(\alpha_n^i)=1,\ \
i \in \mathbb{Z}_{p^n}\right.\right\}\right|.
\end{equation}
According to \eqref{eq6}, in order to find  the linear complexity of
$s^\infty$ over $\mathbb{F}_2$  it is sufficient to find the values
of $T^{(p^n)}_b(x)$ in the set $\{\alpha_n^i,\ \ i =0,1,\dots,
p^n-1\}$.

Now we will study the properties of $T^{(p^j)}_k(x)$.
\begin{lem}
\label{l4} With the notation in \eqref{eq1}, we have  $aD^{(p^j)}_i
 = D^{(p^j)}_{ i+k}$  for $a \in D^{(p^j)}_k$.
\end{lem}

Lemma \ref{l4} may be proved similarly as Lemma 1 from \cite{ZX}.

\begin{lem}
\label{l5} Let $a \in D^{(p^i)}_k$. Then:

(i)
$E^{(p^j)}_{i}(\alpha_j^{p^la})=E^{(p^{j-l})}_{i+k}(\alpha_{j-l})$
for $0\leq l< j$ and $E^{(p^j)}_{i}(\alpha_j^{p^la})=e~(\bmod~2)$
for $l\geq j$;

(ii)
$H^{(p^j)}_{i}(\alpha_j^{p^la})=H^{(p^{j-l})}_{i+k}(\alpha_{j-l})$
for $0\leq l< j$ and
$H^{(p^j)}_{i}(\alpha_j^{p^la})=p^{j-1}(p-1)/2~(\bmod~2)$ for $l\geq
j$;

(iii) $T^{(p^j)}_i(\alpha_j^{p^la})=
H^{(p^{j-l})}_{i+k}(\alpha_{j-l})+H^{(p^{j-l-1})}_{i+k}(\alpha_{j-l-1})+\dots
+H^{(p)}_{i+k}(\alpha_1) + (p^{l}-1)/2~(\bmod~2)$ for $l<j$;

(iv) $T^{(p^j)}_i(\alpha_j^{p^la})= T^{(p^{j-l})}_i(\alpha_{j-l}^a)
+ (p^{l}-1)/2~(\bmod~2)$ for $l<j$.
\end{lem}
\begin{proof} First of all, we prove (i). The assertion in (i) for
$l\geq j$ is clear. Suppose $l<j$; then by Lemma \ref{l4} and
definition of $E^{(p^j)}_{i}$ we see that
$$E^{(p^j)}_{i}(\alpha_j^{p^la})=E^{(p^j)}_{i+k}(\alpha_{j}^{p^l})=\sum_{t\in
D^{(p^j)}_{i+k}}\alpha_{j}^{p^lt}=\sum_{t\in
D^{(p^j)}_{i+k}}\alpha_{j-l}^{t}.$$ For conclusion of proof we note
that $D^{(p^j)}_{i+k} ~(\bmod~p^{j-l})=D^{(p^{j-l})}_{i+k}$ by
\eqref{eq2}.

(ii) By definition and (i) we have
that
$$H^{(p^j)}_k(\alpha_j^{p^la})=\sum_{i=0}^{ d_j/2-1}
E^{(p^j)}_{i+k} (\alpha_j^{p^la})=\sum_{i=0}^{ d_j/2-1}
E^{(p^{j-l})}_{i+k}(\alpha_{j-l})= p^l\sum_{i=0}^{ d_{j-l}/2-1}
E^{(p^{j-l})}_{i+k}(\alpha_{j-l}),$$
 thus  second
statement follows from (i).

(iii) Using definition of $T^{(p^j)}_k(x)$ and (ii), we get
$$T^{(p^j)}_i(\alpha_j^{p^la})=
H^{(p^{j})}_{i}(\alpha_{j}^{p^la})+H^{(p^{j-1})}_{i}(\alpha_{j-1}^{p^la})+\dots
+H^{(p^{l+1})}_{i}(\alpha_{l+1}^{p^la})+H^{(p^{l})}_{i}(1)\dots+H^{(p)}_{i}(1)$$
or
\begin{multline*}
T^{(p^j)}_i(\alpha_j^{p^la})=H^{(p^{j-l})}_{i+k}(\alpha_{j-l})+H^{(p^{j-2})}_{i+k}(\alpha_{j-l-2})+\\
\dots +H^{(p)}_{i+k}(\alpha_1) +
p^{l-1}(p-1)/2+p^{l-2}(p-1)/2+\dots+(p-1)/2~(\bmod~2).
\end{multline*}
Thus, the assertion (iii) of this lemma follows from the latter
identity.

(iv) Observe that $ T^{(p^{j-l})}_i(\alpha_{j-l}^a)
=H^{(p^{j-l})}_{i}(\alpha_{j-l}^a)+H^{(p^{j-l-1})}_{i}(\alpha_{j-l-1}^a)+\dots
+H^{(p)}_{i}(\alpha_1^a)$ we get (iv) by (ii) and (iii).
\end{proof}
Further, since $\alpha_j$ is the primitive $p^j$th root of unity in
an extension field of $\mathbb{F}_2$, we have
$$0=\alpha_j^{p^j}-1=(\alpha_j-1)(\alpha_j^{p^j-1}+\alpha_j^{p^j-2}+\dots+\alpha_j+1).$$
Hence, by \eqref{eq3} and Lemma \ref{l5} we get
\begin{equation}
\label{eq7}
T^{(p^j)}_i(\alpha_j^a)+T^{(p^j)}_{i+d_j/2}(\alpha_j^a)=1
\end{equation}
for any $a: \ \gcd(p,a)=1, j=1,2,\dots, n$.

\begin{lem}
\label{l6} Let  $ 2^{(p-1)}  \not \equiv 1~(\bmod~p^2)$. Then
$T^{(p^n)}_k(\alpha_n^h) \not \in \{0,1\}$ for all $h \in
\mathbb{Z}_{p^n}\setminus p^{n-1}\mathbb{Z}_{p}$ and
$k=0,1,\dots,d_n-1$.
\end{lem}
\begin{proof} We  show that $T^{(p^n)}_k(\alpha_n^h) \not \in \{0,1\}$ by contradiction.
 Suppose  there exists $h=p^{n-m}a$ for $m>1,
\gcd(p,a)=1$ such that $T^{(p^n)}_k(\alpha_n^h) \in \mathbb{F}_2$.
Then by Lemma \ref{l5} (iv) we get $T^{(p^m)}_k(\alpha_m^a) \in
\mathbb{F}_2$.  Without loss of generality by Lemma \ref{l5}, we can
assume that $k=0$, $a=1$ and $T^{(p^m)}_0(\alpha_m)=0$.

 Denote $u=ind_g 2$ and put
$v=\gcd(u,d_m)$.  Then by Lemma \ref{l5} we obtain
$T^{(p^m)}_0(\alpha_m)=\left (T^{(p^m)}_0(\alpha_m)\right )^2=
T^{(p^m)}_{i+u}(\alpha_m)$. Since there  exist integers $y,z:
yu+zd_m=v$, it follows that
\begin{equation}
\label{eq8} 0=T^{(p^m)}_0(\alpha_m)=
T^{(p^m)}_{v}(\alpha_m)=T^{(p^m)}_{2v}(\alpha_m)=\dots =
T^{(p^m)}_{(d_m/v-1)v}(\alpha).
\end{equation}

By \eqref{eq7}  $T^{(p^m)}_{d_m/2}(\alpha_m)=1$, hence $v $ does not
divide  $d_m/2$. Since $v=\gcd(d_m,u)$ and $v \not \equiv
0~(\bmod~p)$, it follows that $v=f=2^r$. With similar arguments as
above we get
$$
1=T^{(p^m)}_{d_m/2}(\alpha_m)=
T^{(p^m)}_{d_m/2+f}(\alpha_m)=T^{(p^m)}_{d_m/2+2f}(\alpha_m)=\dots =
T^{(p^m)}_{d_m/2+(d_m/v-1)f}(\alpha_m).
$$

 So,
$T^{(p^m)}_{0}(\alpha_m)+T^{(p^m)}_{f/2}(\alpha_m)=1$ and by Lemma
\ref{l5} (iii) we have that
$H^{(p^m)}_{0}(\alpha_m)+H^{(p^m)}_{f/2}(\alpha_m)\in
\mathbb{F}_2(\alpha_{m-1}).$
 Denote $\gamma=H^{(p^m)}_{0}(\alpha_m)+H^{(p^m)}_{f/2}(\alpha_m)$.
 Then $\gamma \in
\mathbb{F}_2(\alpha_{m-1})$ and
\begin{equation}
\label{eq9}
E^{(p^m)}_{0}(\alpha_m)+\dots+E^{(p^m)}_{f/2-1}(\alpha_m)+E^{(p^m)}_{fp^{m-1}/2}(\alpha_m)
+\dots+E^{(p^m)}_{f/2+fp^{m-1}/2-1}(\alpha_m)=\gamma\end{equation}

Put
$C=D^{(p^m)}_{0}+\dots+D^{(p^m)}_{f/2-1}+D^{(p^m)}_{fp^{m-1}/2}+\dots+D^{(p^m)}_{f/2+fp^{m-1}/2-1}$.

Denote $\bar{c}=c~(\bmod~p)$ and define
$$
f(x)=\sum_{c\in C}
x^{\bar{c}}\alpha_{m-1}^{\frac{c-\bar{c}}{p}}+\gamma.
$$
 By
\eqref{eq2} we get
\begin{multline*}C~(\bmod~p)=D^{(p^m)}_{0}~(\bmod~p)+\dots+D^{(p^m)}_{f/2-1}~(\bmod~p)+\\
D^{(p^m)}_{fp^{m-1}/2}~(\bmod~p)
+\dots+D^{(p^m)}_{f/2+fp^{m-1}/2-1}~(\bmod~p)=\\
D^{(p)}_{0}+\dots+D^{(p)}_{f/2-1}+D^{(p)}_{f/2}
+\dots+D^{(p)}_{f-1}= \mathbb{Z}^{*}_p.\end{multline*}
 Thus
$f(x)=\sum_{i=1}^{p-1} c_i x^{i}+\gamma$ and $c_i\in
\mathbb{F}_2(\alpha_{m-1})$. By  \eqref{eq9} we obtain that
$$
f(\alpha_{m})=\sum_{c\in C}
\alpha_{m-1}^{\bar{c}}\alpha_{m-1}^{\frac{c-\bar{c}}{p}}+\gamma=\sum_{c\in
C} \alpha_{m}^{\bar{c}}\alpha_{m}^{{c-\bar{c}}}+\gamma=0.
$$
So, $\alpha_{m}$ is a root of polynomial $f(x)$ with coefficients
from $\mathbb{F}_2(\alpha_{m-1})$ and $\deg
 f(x)=p-1$, hence
$[\mathbb{F}_2(\alpha_{m}):\mathbb{F}_2(\alpha_{m})]<p$. This
contradicts Lemma \ref{l3}.
\end{proof}

\subsection{Main result}
After the preparations in Sect. 3.2, we can now prove the main
result of this paper.
\begin{thm}
\label{t1} Let $p$ be an odd prime and let $e$ be a factor of $p-1$
such that $f = (p-1)/ e$ is of the form $2^r$ with $r \geq 1$. Let
$s^\infty$ be a generalized cyclotomic binary sequence of period
$p^n$ defined by \eqref{eq4}. If  $2^{(p-1)} \not \equiv
1~(\bmod~p^2)$, then
$$ L =\begin{cases}
 p^n-\frac{p-1}{2}-\delta\left (\frac{p^n+1}{2}\right) ,&\text{ if  }   2 \in D^{(p)}_0, \\
p^n-\delta\left (\frac{p^n+1}{2}\right) ,&\text{ if  }   2 \not \in D^{(p)}_0, \\
 \end{cases}$$
where $\delta(t) = 1$ if $t$ is even and $\delta(t) = 0$ if $t$ is
odd.
\end{thm}

\begin{proof}
  By definition of this sequence and \eqref{eq3} $S(1)=(p^n+1)/2$.
 By Lemmas \ref{l5} and \ref{l6} we get that
$T^{(p^n)}_b(\alpha_n^h) \not \in \{0,1\}$ for all $h \in
\mathbb{Z}_{p^n}\setminus p^{n-1}\mathbb{Z}_{p}$.

Further, if $h \in p^{n-1}\mathbb{Z}_{p}^*$ then by Lemma \ref{l5}
we see that $T^{(p^n)}_i(\alpha_n^{p^{n-1}a}) \in \{0, 1\}$  for $a
\in \mathbb{Z}_p^*$ if and only if $T^{(p)}_i(\alpha_1^{a}) \in \{0,
1\}$. The properties of polynomial $T^{(p)}_i(x)$ were studied in
\cite{ZX}. By \cite{ZX}, if $T^{(p)}_b(\alpha_1^{a}) \in \{0, 1\}$
then $2 \in D^{(p)}_0$, $|\{a: \ T^{(p)}_b(\alpha_1^{a})=1,
a=1,2,\dots, p-1\}=(p-1)/2$ and vise versa.

So, $$\left|\left\{i \left| T^{(p^n)}_b(\alpha_n^i)=1,\ \ i \in
\mathbb{Z}_{p^n}\right.\right\}\right|= \begin{cases} \frac{p-1}{2}+\delta\left (\frac{p^n+1}{2}\right) ,&\text{ if  }   2 \in D^{(p)}_0, \\
\delta\left (\frac{p^n+1}{2}\right) ,&\text{ if  }   2 \not \in D^{(p)}_0, \\
 \end{cases}$$
where $\delta(t) = 1$ if $t$ is even and $\delta(t) = 0$ if $t$ is
odd. This completes the proof of Theorem \ref{t1}.
\end{proof}
If prime $p$ satisfying $2^{(p-1)}  \equiv 1~(\bmod~p^2)$ then $p$
is called a Wieferich prime. By \cite{DK}  there are only two
Wieferich primes, 1093 and 3511, up to $6.7 \times 10^{15}$.
Therefore, for all other primes $p<6.7 \times 10^{15}$ besides 1093
and 3511 the result in Theorem \ref{t1} holds.

\section*{Conclusion}
In this paper, the linear complexity  of a class of new binary
cyclotomic sequences of period $p^n$ was studied. These almost
balanced  sequences are constructed by cyclotomic classes  using a
method presented by Z. Xiao et al. We proved their conjecture about
 linear complexity of these sequences and show  it be high  when $p: \ 2^{(p-1)} \not \equiv
1~(\bmod~p^2)$. In this case the linear complexity of sequence is
very close to the period. It is an interesting  problem to study
other characteristics of these sequences and also the linear
complexity for $f\neq 2^r$.

\end{document}